\documentclass{article}
\usepackage{graphicx} 
\usepackage{amsthm}
\usepackage[normalem]{ulem} 
\usepackage{xcolor} 
\usepackage{asymptote}
\usepackage[colorlinks=true]{hyperref}

\title{A note on extendable sets of colorings and rooted minors}
\author{Zden\v{e}k Dvo\v{r}\'ak\thanks{Charles University, Prague, Czech Republic, \url{rakdver@iuuk.mff.cuni.cz}.  Supported by the
        ERC-CZ project LL2328 (Beyond the Four Color Theorem, registered as ``Realizace projektu hraničního výzkumu v oblasti
        teorie grafů -- barevnost a návrh algoritmů'') of the Ministry of Education of Czech Republic.} \and Jan M.\ Swart\thanks{The Czech Academy of Sciences, Institute of Information Theory and Automation, Prague, Czech Republic, \url{swart@utia.cas.cz}.}}
\date{January 2025}

\newtheorem{theorem}{Theorem}
\newtheorem{corollary}[theorem]{Corollary}

\newtheorem{observation}[theorem]{Observation}
\newtheorem{conjecture}[theorem]{Conjecture}
\newcommand{\Fcopy}[1]{F_{\mathrm{copy},#1}}
\newcommand{\Fenc}[1]{F_{\mathrm{enc},#1}}

\begin{document}

\maketitle

\begin{abstract}
DeVos and Seymour (2003) proved that for every set $C$ of 3-colorings of a set $X$ of
vertices, there exists a plane graph $G$ with vertices of $X$ incident with the outer
face such that a 3-coloring of $X$ extends to a 3-coloring of $G$ if and only if it belongs to $C$.
We prove a generalization of this claim for $k$-colorings of $X$-rooted-$K_{k+1}$-minor-free $K_{k+2}$-minor-free graphs.
\end{abstract}

\section{Introduction}

For a positive integer $k$, a \emph{set of $k$-colorings} of a set $X$
of vertices is a set $C$ of functions $f:X\to [k]$ closed under permutations
of colors, i.e., for every permutation $\pi:[k]\to [k]$, if $f\in C$,
then $\pi\circ f\in C$.  For a graph $G$, we say that $G$ \emph{realizes} $C$
if $X\subseteq V(G)$ and each function $f:X\to [k]$ belongs to $C$ if and only if $f$ extends to a proper $k$-coloring of $G$.
Suppose that the set $X$ is cyclically ordered.  In this case, we say that $C$ is \emph{planarly realizable} if there
exists a plane graph $G$ realizing $C$ such that all vertices of $X$
are incident with the outer face of $G$ and the order of appearance of the vertices of $X$ in the boundary of the outer
face of $X$ matches the prescribed cyclic ordering of $X$.

The proofs of many coloring results in planar graphs, including the proof of the Four Color Theorem,
depend on the following fact: For every $k\ge 4$, there exist many sets of $k$-colorings that
are not planarly realizable (indeed, the planarly realizable sets must satisfy Kempe chain constraints).
Briefly, these proofs are based on the existence of reducible configurations.  Let $K$ be a plane graph with a specified set $X$ of vertices incident with the outer face of $K$,
let $X$ be cyclically ordered in the order the vertices appear along the boundary of the outer face of $K$,
and let $C_K$ be the set of all $k$-colorings of $X$ that extend to a $k$-coloring of $K$.
In the simplest form, we say that $(K,X)$ is a \emph{reducible configuration} if every non-empty planarly realizable set of $k$-colorings of $X$ intersects $C$,
or equivalently, no non-empty subset of the complement of $C$ is planarly realizable.

The motivation for this definition is as follows. Suppose that $G$ is a plane graph and $\Delta$ is a closed disk in the plane whose boundary intersects the drawing of $G$ in a set $Y$ of vertices.
Let $G_1$ be the subgraph of $G$ drawn in $\Delta$ and $G_2$ the subgraph of $G$ drawn in the closure of the complement of $\Delta$.
If $(G_1,Y)$ is isomorphic to $(K,X)$, we say that the configuration $(K,X)$ \emph{appears} in $G$.  Since $(K,X)$ is a reducible configuration, it follows that 
$G$ is $k$-colorable if and only if $G_2$ is $k$-colorable.  This reduces the problem of proving that $G$ is $k$-colorable to the smaller graph $G_2$.

On the other hand, DeVos and Seymour~\cite{DS03} proved that the situation is different for $k=3$.
\begin{theorem}[DeVos and Seymour~\cite{DS03}]\label{thm-real3}
For every cyclically ordered set $X$ of vertices, every set of $3$-colorings of $X$ is planarly realizable.
\end{theorem}
Thus, for $3$-coloring, $(K,X)$ is a reducible configuration only if all 3-colorings of $X$ extend to $K$, making the reducible
confgurations for $3$-coloring much less prevalent.  Of course, there are other variants of reducibility (e.g., instead of fully removing $K$,
we can replace it by a smaller subgraph), but they are similarly affected.

The influential Hadwiger's conjecture states that for every positive integer $k$, every $K_{k+1}$-minor-free graph is $k$-colorable.
The conjecture has only been confirmed for $k\le 5$, see~\cite{RST93}, but it has sparked a lot of interest in the theory
of graph minors and in particular in their coloring properties.  Thus, it is important to consider the limits of the method
of reducible configurations in this context as well.

A \emph{model} of a graph $H$ in a graph $G$ is a function
$\mu$ that assigns pairwise vertex-disjoint connected subgraphs of $G$ to vertices
of $H$ so that for every edge $uv\in E(H)$, there exists an edge of $G$
with one end in $\mu(u)$ and the other end in $\mu(v)$.  If $H$ has a model in $G$,
then $H$ is a \emph{minor} of $G$; otherwise, we say that $G$ is \emph{$H$-minor-free}.
For a set $X\subseteq V(G)$, we say that the model is \emph{$X$-rooted} if
$V(\mu(v))\cap X\neq \emptyset$ for all $v\in V(H)$.  
If $H$ has an $X$-rooted model in $G$, we say that $H$ is an \emph{$X$-rooted minor} of $G$;
otherwise, we say that $G$ is \emph{$X$-rooted-$H$-minor-free}.

Since the graph $K_4$ is not outerplanar and $K_5$ is not planar, Theorem~\ref{thm-real3} has the following consequence.
\begin{corollary}\label{cor-real3}
For every set $X$ of vertices and every set $C$ of $3$-colorings of $X$, there exists
an $X$-rooted-$K_4$-minor-free $K_5$-minor-free graph $G$ realizing $C$.
\end{corollary}

In this note, we show that this statement generalizes as follows.
\begin{theorem}\label{thm-main}
For every integer $k\ge 3$, every set $X$ of vertices and every set $C$ of $k$-colorings of $X$, there exists
an $X$-rooted-$K_{k+1}$-minor-free $K_{k+2}$-minor-free graph $G$ realizing $C$.
\end{theorem}
Thus, similarly to Theorem~\ref{thm-real3}, reducible configurations in general cannot exploit the global properties when the number of
colors is smaller than what is postulated by Hadwiger's conjecture.

Holroyd~\cite{Hol97} proposed the following strengthening of Hadwiger's conjecture: If $G$ is $k$-colorable and every $k$-coloring of $G$ uses all $k$ colors on a set $X$ of vertices,
then $K_k$ is an $X$-rooted minor of $G$.  Holroyd proved this to be true for $k=3$, and Martinsson and Steiner~\cite{MS24} for $k=4$.
If true in general, this conjecture implies that Theorem~\ref{thm-main} does not extend to $X$-rooted-$K_k$-minor-free graphs, i.e., the following conjecture is a weakening of Holroyd's one.
\begin{conjecture}
For every integer $k\ge 3$, there exists a set $X$ of vertices and a set $C$ of $k$-colorings of $X$ such that
no $X$-rooted-$K_k$-minor-free graph $G$ realizes~$C$.
\end{conjecture}

\section{Proof}

Our proof depends on the result of DeVos and Seymour (Theorem~\ref{thm-real3}). Since the case $k=3$ is Corollary~\ref{cor-real3}, we assume $k\geq 4$ from now on.

We are going to need several gadgets.  The first one is a ``copy gadget'' $\Fcopy{k}(u,v)$, obtained from the clique on $k+1$ vertices by deleting an edge $uv$.  Clearly,
$f:\{u,v\}\to[k]$ extends to a $k$-coloring of $\Fcopy{k}(u,v)$ if and only if $f(u)=f(v)$.
Let $\Fcopy{k}^+=\Fcopy{k}(u,v)+uv$ be the clique of size $k+1$.

The second gadget is more complicated. For $s\in \{4,\ldots, k\}$, we are going to construct $\Fenc{s,k}(u,v,w;y_4,\ldots,y_k)$ as a graph $F$ that contains a set $A=\{u,v,w,y_4,\ldots,y_k\}$ of distinct vertices and
the following claim holds: Let $f:A\to [k]$ be a function such that $f(y_i)=i$ for all $i\in \{4,\ldots,k\}$. Then $f$ extends to a $k$-coloring of $F$ if and only if
\begin{itemize}
\item $f(u)\in \{1,2,3\}$ and $f(u)=f(v)=f(w)$, or
\item $f(u)=s$, $f(v),f(w)\in \{1,2,3\}$, and $f(v)\neq f(w)$, or
\item $f(u)\in \{4,\ldots,k\}\setminus\{s\}$ and $f(v)=f(w)\in\{1,2,3\}$.
\end{itemize}
\begin{figure}
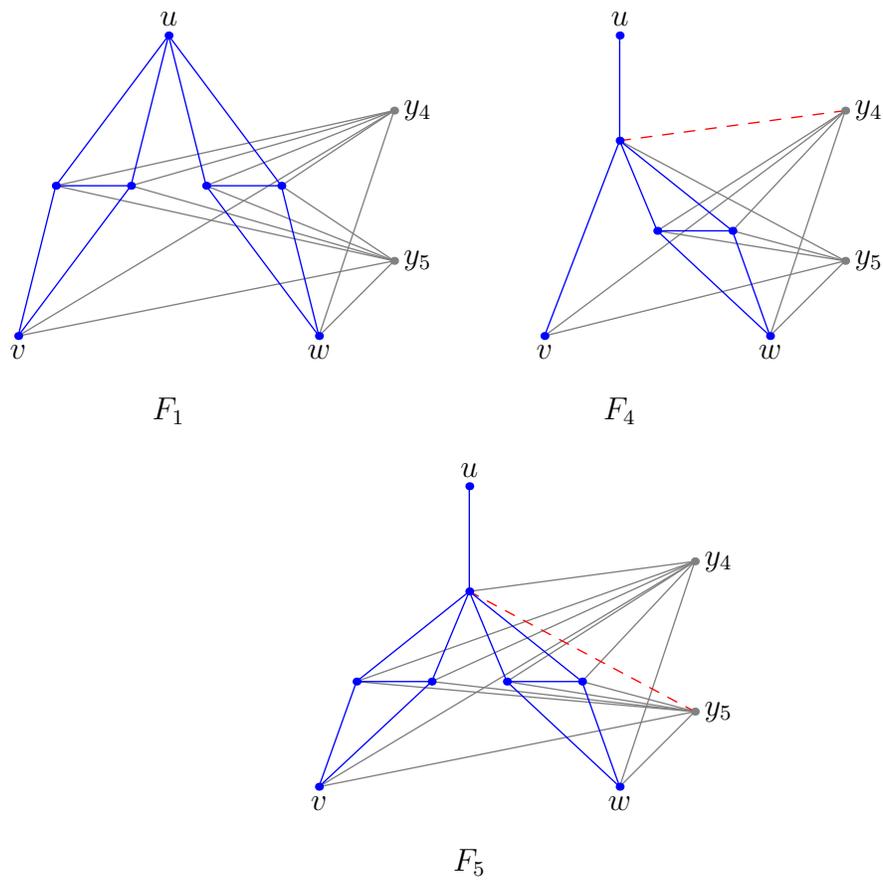

\begin{asy}
unitsize (2cm);

pair v[];

v[0] = (0,0);
v[1] = v[0] + (-1,-2);
v[2] = v[0] + (1,-2);
v[3] = v[0] + (-0.75,-1);
v[4] = v[0] + (-0.25,-1);
v[5] = v[0] + (0.25,-1);
v[6] = v[0] + (0.75,-1);
v[7] = v[0] + (1.5,-0.5);
v[8] = v[0] + (1.5,-1.5);

for (int i = 1; i <= 6; ++i)
  for (int j = 7; j <= 8; ++j)
    draw (v[i]--v[j], 0.5white);
draw(v[0]--v[3]--v[1]--v[4]--v[0]--v[5]--v[2]--v[6]--v[0], blue);
draw(v[3]--v[4], blue);
draw(v[5]--v[6], blue);
for (int i = 0; i <= 6; ++i)
  dot (v[i], blue);
for (int i = 7; i <= 8; ++i)
  dot (v[i], 0.5white);

label ("$u$", v[0], N);
label ("$v$", v[1], S);
label ("$w$", v[2], S);
label ("$y_4$", v[7], E);
label ("$y_5$", v[8], E);
label ("$F_1$", v[0] + (0, -2.5));

v[0] = (3,0);
v[1] = v[0] + (-0.5,-2);
v[2] = v[0] + (1,-2);
v[3] = v[0] + (0,-0.7);
v[4] = v[0] + (0.25,-1.3);
v[5] = v[0] + (0.75,-1.3);
v[6] = v[0] + (1.5,-0.5);
v[7] = v[0] + (1.5,-1.5);

for (int i = 1; i <= 5; ++i)
  for (int j = 6; j <= 7; ++j)
    if (i != 3 || j != 6)
      draw (v[i]--v[j], 0.5white);
draw(v[0]--v[3]---v[4]--v[2]--v[5]--v[4], blue);
draw(v[5]--v[3]--v[1], blue);

draw (v[3] -- v[6], red + dashed);

for (int i = 0; i <= 5; ++i)
  dot (v[i], blue);
for (int i = 6; i <= 7; ++i)
  dot (v[i], 0.5white);

label ("$u$", v[0], N);
label ("$v$", v[1], S);
label ("$w$", v[2], S);
label ("$y_4$", v[6], E);
label ("$y_5$", v[7], E);
label ("$F_4$", v[0] + (0, -2.5));

v[0] = (2,-3);
v[1] = v[0] + (-1,-2);
v[2] = v[0] + (1,-2);
v[3] = v[0] + (-0.75,-1.3);
v[4] = v[0] + (-0.25,-1.3);
v[5] = v[0] + (0.25,-1.3);
v[6] = v[0] + (0.75,-1.3);
v[7] = v[0] + (0,-0.7);
v[8] = v[0] + (1.5,-0.5);
v[9] = v[0] + (1.5,-1.5);

for (int i = 1; i <= 7; ++i)
  for (int j = 8; j <= 9; ++j)
    if (i != 7 || j != 9)
      draw (v[i]--v[j], 0.5white);

draw (v[7] -- v[9], red + dashed);

draw(v[0]--v[7]--v[3]--v[1]--v[4]--v[7]--v[5]--v[2]--v[6]--v[7], blue);
draw(v[3]--v[4], blue);
draw(v[5]--v[6], blue);
for (int i = 0; i <= 7; ++i)
  dot (v[i], blue);
for (int i = 8; i <= 9; ++i)
  dot (v[i], 0.5white);

label ("$u$", v[0], N);
label ("$v$", v[1], S);
label ("$w$", v[2], S);
label ("$y_4$", v[8], E);
label ("$y_5$", v[9], E);
label ("$F_5$", v[0] + (0, -2.5));

\end{asy}
\caption{Pieces of the gadget $\Fenc{4,5}(u,v,w;y_4,y_5)$.}\label{fig-gadget}
\end{figure}
Thus, this gadget encodes information about the color of $u$ in the 3-coloring of $v$ and $w$: The vertices $v$ and $w$ have different colors if and only if $u$ has color $s$, and if the color of $u$ belongs to $\{1,2,3\}$, it is copied to the colors of $v$ and $w$.  We construct $F$ from several pieces, see Figure~\ref{fig-gadget}:
\begin{itemize}
\item Let $F'_1$ be graph with vertex set $\{u,v,w,v_1,v_2,w_1,w_2\}$ and edges $uv_i$, $v_iv$, $uw_i$, $w_iw$ for $i\in\{1,2\}$, $v_1v_2$, and $w_1w_2$.
Let $F_1$ be obtained from $F'_1$ by adding the vertices $y_4$, \ldots, $y_k$ adjacent to all vertices of $V(F'_1)\setminus\{u\}$.
Note that $f$ extends to a $k$-coloring of $F_1$ if and only if either
\begin{itemize}
\item $f(u)\in\{4,\ldots,k\}$ and $f(v),f(w)\in\{1,2,3\}$, or
\item $f(u)=f(v)=v(w)\in\{1,2,3\}$.
\end{itemize}
\item Let $F'_2$ be the graph with vertex set $\{u,v,w,u',w_1,w_2\}$ and edges $u'w_i$, $w_iw$ for $i\in\{1,2\}$, $uu'$, $w_1w_2$, and $u'v$.
Let $F''_2$ be obtained from $F'_2$ by adding the vertices $y_4$, \ldots, $y_k$ adjacent to all vertices of $V(F'_2)\setminus\{u\}$.
Let $F_s=F''_2-u'y_s$.  Note that $f$ extends to a $k$-coloring of $F_s$ if and only if either
\begin{itemize}
\item $f(u)\neq s$ and $f(v),f(w)\in\{1,2,3\}$, or
\item $f(u)=s$, $f(v),f(w)\in\{1,2,3\}$, and $f(v)\neq f(w)$.
\end{itemize}
\item Let $F'_3$ be the graph with vertex set $\{u,v,w,u',v_1,v_2,w_1,w_2\}$
and edges $u'v_i$, $v_iv$, $u'w_i$, $w_iw$ for $i\in\{1,2\}$, $uu'$, $v_1v_2$, and $w_1w_2$.
Let $F''_3$ be obtained from $F'_3$ by adding the vertices $y_4$, \ldots, $y_k$ adjacent to all vertices of $V(F'_3)\setminus\{u\}$.
For $r\in \{4,\ldots,k\}\setminus\{s\}$, let $F_r=F''_3-u'y_r$.  Note that $f$ extends to a $k$-coloring of $F_r$ if and only if either
\begin{itemize}
\item $f(u)\neq r$ and $f(v),f(w)\in\{1,2,3\}$, or
\item $f(u)=r$ and $f(v)=f(w)\in \{1,2,3\}$.
\end{itemize}
\end{itemize}
Then $F$ is obtained from the disjoint union of copies of $F_1$, $F_4$, \ldots, $F_k$ by identifying the copies of the vertices $A$.
Let $\Fenc{s,k}^+$ be obtained from $F$ by adding all edges of a clique on $\{u,v,w;y_4,\ldots,y_k\}$.

\begin{observation}\label{obs-mifr}
For all integers $k\ge 4$ and $s\in\{4,\ldots,k\}$, the graphs $\Fcopy{k}^+$ and $\Fenc{s,k}^+$ are $K_{k+2}$-minor-free.
\end{observation}
\begin{proof}
This is clear for $\Fcopy{k}^+=K_{k+1}$.  The graph of $\Fenc{s,k}^+$ can be obtained as follows:  We start with the planar
graphs $F'_1+uvw$, $F'_2+uvw$, and $k-4$ copies of $F'_3+uvw$, which are $K_5$-minor-free.  We take the clique-sum of
these graphs on the triangle $uvw$, obtaining another $K_5$-minor-free graph.  Finally, we add $k-3$ vertices $y_4$, \ldots, $y_k$
and the incident edges, resulting in a $K_{k+2}$-minor-free graph.
\end{proof}

We are now ready to prove the main result.
\begin{figure}
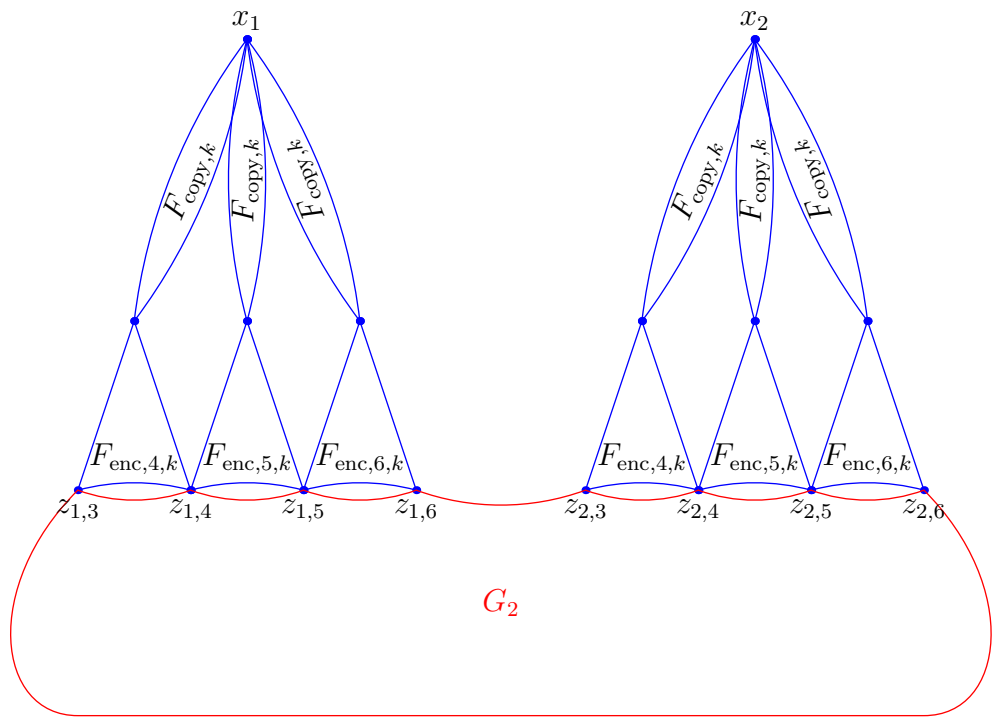

\begin{asy}
unitsize (1.5cm);

void copy_gadget (pair f, pair t)
{
  dot (f, blue);
  dot (t, blue);
  draw (f{dir(degrees(t-f) + 15)} .. {dir(degrees(t-f) - 15)}t{dir(degrees(f-t) + 15)} .. {dir(degrees(f-t) - 15)}cycle, blue);
  label (rotate(degrees(f-t)) * Label("$F_{\mathrm{copy},k}$"), (f+t)/2);
}

void enc_gadget (pair u, pair v, pair w, int s)
{
  draw (u -- v{dir(15)} .. {dir(-15)}w--cycle, blue);
  label (Label("$F_{\mathrm{enc}," + (string) s + ",k}$"), (u+v+w)/3, S);
  dot (u, blue);
  dot (v, blue);
  dot (w, blue);
}

pair x[], xp[], z[];

for (int i = 0; i < 2; ++i)
  {
    x[i] = (4.5i,0);
    for (int j = 0; j <= 2; ++j)
      {
        xp[3i+j] = x[i] + (-1 + j, -2.5);
	copy_gadget (x[i],xp[3i+j]);
      }	
    for (int j = 0; j <= 3; ++j)
      z[4i+j] = x[i] + (-1.5 + j, -4);
    for (int j = 0; j <= 2; ++j)
      enc_gadget (xp[3i+j],z[4i+j],z[4i+j+1], j+4);
  }

for (int i = 0; i < 2 * 4 - 1; ++i)
  draw(z[i]{dir(-20)}..{dir(20)}z[i+1], red);

draw (z[0]{dir(-135)} .. {dir(0)}(z[0]+(0,-2)) -- (z[7]+(0,-2)){dir(0)}..{dir(135)}z[7], red);
label ("$G_2$", (z[0]+z[7])/2 + (0,-1), red);

for (int i = 0; i < 2; ++i)
  {
    label ("$x_" + (string) (i+1) + "$", x[i], N);
    for (int j = 0; j <= 3; ++j)
      label ("$z_{" + (string) (i+1) + "," + (string) (j+3) + "}$", z[4i+j], S);
  }
\end{asy}
\caption{The construction from the proof of Theorem~\ref{thm-main}, the case $m=2$ and $k=6$.  The apex vertices $y_4$, \ldots, $y_k$ are not shown.}\label{fig-constr}
\end{figure}

\begin{proof}[Proof of Theorem~\ref{thm-main}]
Let $X=\{x_1,\ldots,x_m\}$, and let $Y=\{y_4,\ldots,y_k\}$ and $Z=\{z_{i,j}:i\in[m],j\in\{3,\ldots,k\}\}$ be sets of distinct vertices. Given a function $f:X\to[k]$, let $C_f$ be the set of functions $g:Z\to [3]$ such that for each $i\in[m]$,
\begin{itemize}
\item if $f(x_i)\in \{1,2,3\}$, then $g(z_{i,3})=g(z_{i,4})=\ldots=g(z_{i,k})=f(x_i)$, and
\item if $f(x_i)=s\in\{4,\ldots,k\}$, then $g(z_{i,3})=\ldots=g(z_{i,s-1})\neq g(z_{i,s})=\ldots =g(z_{i,k})$.
\end{itemize}
A function $h:X\cup Y\cup Z\to [k]$ is \emph{$Y$-rainbow} if $h(y_i)=i$ for $i\in \{4,\ldots,k\}$.  We are going to construct a graph $G_1$ with $X\cup Y\cup Z\subseteq V(G_1)$ such that every $Y$-rainbow function $h$ can be extended to a $k$-coloring of $G_1$ if and only if its
restrictions $f$ and $g$ to $X$ and $Z$, respectively, satisfy $g\in C_f$. Let us remark that this implies that the restriction of $h$ to $Z$ is a 3-coloring. Thus, for each $i\in[m]$, the graph $G_1$ uniquely encodes the color of $x_i$ in the 3-coloring of $\{z_{i,3},\ldots,z_{i,k}\}$.

We construct $G_1$ as follows, see the top part of Figure~\ref{fig-constr}. We start with the set
of vertices $X\cup Y\cup Z$. For $i\in [m]$ and $j\in \{4,\ldots, k\}$, we add a
vertex $x_{i,j}$ and a copy of the gadget $\Fcopy{k}(x_i,x_{i,j})$. For
$i\in[m]$ and $s\in \{4,\ldots,k\}$, we add a copy of the
gadget $\Fenc{s,k}(x_{i,s},z_{i,s-1},z_{i,s};y_4,\ldots,y_k)$. Then $G_1$ has
the property just described.

Let $C$ be the set of $k$-colorings of $X$ that we want to realize and let $C'=\bigcup_{f\in C} C_f$. Observe that $C'$ is a set of $3$-colorings of $Z$. Let $G'_2$ be a plane graph with the vertices
\[
z_{1,3}, z_{1,4}, \ldots, z_{1,k},z_{2,3}, \ldots z_{2,k},\ldots, z_{m,3},\ldots, z_{m,k}
\]
drawn in the boundary of the outer face of $G'_2$ in order such that $G'_2$
realizes $C'$; such a plane graph exists by Theorem~\ref{thm-real3}. Let $G_2$
be the graph obtained from $G'_2$ by adding $Y$ and all edges between $Y$ and
$G'_2$ as well as the edges of the clique on $Y$. We define $G$ to be the union
of $G_1$ and $G_2$. We claim that $G$ realizes $C$.

Indeed, for any $k$-coloring $f\in C$, we can choose a $Y$-rainbow function $h:X\cup Y\cup Z\to [k]$ that extends $f$ and
whose restriction $g$ to $Z$ satisfies $g\in C_f$. Such a function can be
extended to a $k$-coloring of both $G_1$ and $G_2$, and hence also of $G$.

Conversely, suppose that 
$h:X\cup Y\cup Z\to[k]$ extends
to a $k$-coloring of $G$, and let $f$ be the restriction of $h$ to $X$.  We need to show that $f\in C$.  Since $C$ is a set of $k$-colorings, we can permute the colors and without loss of generality assume that $h$ is $Y$-rainbow.
By the choice of $G_2$, the restriction
$g$ of $h$ to $Z$ is a $3$-coloring and $g\in C'$, and by the choice of $G_1$, we have $g\in C_f$. It follows that $C_f\cap C'\neq\emptyset$, and thus
$f\in C$ by the choice of $C'$.

It remains to argue that $G$ is $X$-rooted-$K_{k+1}$-minor-free and $K_{k+2}$-minor-free. Observe that $G$ can be obtained as follows.  We start with the plane graph $G_0$ consisting of $G'_2$ and and the graph drawn in Figure~\ref{fig-constr} in blue (i.e., the triangles $x_{i,j}z_{i,j-1}z_{i,j}$ and the edges $x_ix_{i,j}$ for $i\in[m]$ and $j\in\{4,\ldots,k\}$).
Note that the vertices of $X$ are drawn in the boundary of the outer face of $G_0$, and thus
$G_0$ is $X$-rooted-$K_4$-minor-free and $K_5$-minor-free.  We then add $k-3$ universal vertices of $Y$, obtaining a graph
which is $X$-rooted-$K_{k+1}$-minor-free and $K_{k+2}$-minor free.  Finally, we perform clique-sums of this graph with copies of  graphs $\Fcopy{k}^+$ and $\Fenc{\star,k}^+$.
Since the added graphs are $K_{k+2}$-minor-free and all vertices of $X$ are contained in $G_0$, it follows that the resulting graph $G$ is also $X$-rooted-$K_{k+1}$-minor-free and $K_{k+2}$-minor-free.
\end{proof}

\end{document}